\documentclass[a4paper,reqno]{amsart}

\usepackage[utf8]{inputenc}
\usepackage{amsmath,amssymb} 

\usepackage{csquotes}
\usepackage{xcolor}
\usepackage{mathrsfs}
\usepackage{bbm}

\theoremstyle{plain}
\newtheorem{theorem}{Theorem}[section]
\newtheorem*{theorem*}{Theorem}
\newtheorem{corollary}[theorem]{Corollary}
\newtheorem{proposition}[theorem]{Proposition}
\newtheorem{lemma}[theorem]{Lemma}
\theoremstyle{definition}

\newtheorem*{definition*}{Definition}
\newtheorem*{remark}{Remark}

\newcommand{\RR}{\mathbb{R}}
\newcommand{\CC}{\mathbb{C}}

\newcommand{\NN}{\mathbb{N}}
\newcommand{\D}{\mathcal{D}}
\newcommand{\N}{\mathcal{N}}
\newcommand{\R}{\mathcal{R}}

\newcommand{\DDV}{{\Delta_\mathcal{D}^V}}
\newcommand{\DNV}{{\Delta_\mathcal{N}^V}}
\newcommand{\DRV}{{\Delta_\mathcal{R}^V}}

\newcommand{\aDV}{{\mathfrak{a}_\mathcal{D}^V}}
\newcommand{\aNV}{{\mathfrak{a}_\mathcal{N}^V}}
\newcommand{\aRV}{{\mathfrak{a}_\mathcal{R}^V}}
\newcommand{\normalderiv}[2][\partial\Omega]{{\mathchoice%
    {\frac{\partial #2}{\partial\nu}\Big|_{#1}}
    {\frac{\partial #2}{\partial\nu}\big|_{#1}}
    {\frac{\partial #2}{\partial\nu}|_{#1}}
    {\frac{\partial #2}{\partial\nu}|_{#1}}
}}

\newcommand{\norm}[1]{{\left\vert\kern-0.25ex\left\vert #1 \right\vert\kern-0.25ex\right\vert}}
\newcommand{\scalar}[2]{{\left(#1,#2\right)}}
\newcommand{\duality}[2]{{\left<#1,#2\right>}}
\newcommand{\LTwoScalar}[3]{{\left(#1,#2\right)_{L^2(#3)}}}

\renewcommand{\Re}{\textnormal{Re}}

\renewcommand{\emptyset}{\varnothing}

\newcommand{\cA}{\mathcal{A}}
\newcommand{\cL}{\mathcal{L}}
\newcommand{\cO}{\mathcal{O}}
\newcommand{\fa}{\mathfrak{a}}

\DeclareMathOperator{\dom}{dom}
\DeclareMathOperator{\ran}{ran}

\numberwithin{equation}{section}

\title[]{Eigenvalue inequalities for Schrödinger operators on unbounded Lipschitz domains}

\author{Jussi Behrndt}
\address{Institut f\"ur Numerische Mathematik \\
    Technische Universit\"at Graz \\
    Steyrergasse 30\\
    A-8010 Graz\\
    Austria}

\author{Jonathan Rohleder}

\author{Simon Stadler}

\email{behrndt@tugraz.at, rohleder@tugraz.at, simon.stadler@student.tugraz.at}


\begin{document}
\begin{abstract}
Given a Schr\"odinger differential expression on an exterior Lipschitz domain we prove strict inequalities between the eigenvalues of the corresponding 
selfadjoint operators subject to Dirichlet and Neumann or Dirichlet and mixed boundary conditions, respectively. Moreover, we prove a strict inequality 
between the eigenvalues of two different elliptic differential operators on the same domain with Dirichlet boundary conditions.
\end{abstract}
\maketitle
    
    
\section{Introduction}

In the spectral theory of Laplace and Schr\"odinger operators eigenvalue inequalities have a long history, see, e.g.,~\cite{AB07} for a survey. 
One extensively studied topic is the relation between Dirichlet and Neumann eigenvalues for the Laplace equation on a bounded domain; we refer 
the reader to the classical works~\cite{P55, P52, S54} and the more recent contributions~\cite{AM12,F05,F91,F95,GM09,LW86,M91,R14,S08}. 

In this note we focus on eigenvalue 
inequalities for Schr\"odinger operators 
on exterior domains with Dirichlet, Neumann, and Robin boundary conditions. As a special case consider first the selfadjoint Schr\"odinger operators
\begin{align*}
 - \DDV u = - \Delta u + V u, \quad \dom (- \DDV) = \big\{ u \in H^1 (\Omega): \Delta u \in L^2 (\Omega), u |_{\partial \Omega} = 0 \big\},
\end{align*}
and
\begin{align*}
 - \DNV u = - \Delta u + V u, \quad \dom (- \DNV) = \Big\{ u \in H^1 (\Omega): \Delta u \in L^2 (\Omega), \frac{\partial u}{\partial \nu} \Big|_{\partial \Omega} = 0 \Big\},
\end{align*}
in $L^2 (\Omega)$ on an exterior Lipschitz domain $\Omega \subset \RR^d$, $d \geq 2$, with a bounded, measurable potential 
$V : \Omega \to \RR$; here $u |_{\partial \Omega}$ and $\frac{\partial u}{\partial \nu} |_{\partial \Omega}$ are the trace and the normal derivative of 
a function $u\in H^1(\Omega)$, respectively. 
The essential spectra of $- \DDV$ and $- \DNV$ coincide and, depending 
on the form of the potential $V$, there may exist finitely or infinitely many eigenvalues below the bottom of the essential spectrum. We denote these eigenvalues by
\begin{align*}
 \lambda_1^\D \leq \lambda_2^\D \leq \dots \quad \text{and} \quad \lambda_1^\N \leq \lambda_2^\N \leq \dots,
\end{align*}
respectively, if they are present. It follows immediately from variational principles that if $- \DDV$ possesses (at least) 
$l$ eigenvalues below the bottom of the essential spectrum then the same is true for $- \DNV$ and the inequality
\begin{align}\label{eq:langweiligeUngleichung}
 \lambda_k^\N \leq \lambda_k^\D, \qquad k \in \{1, \dots, l\},
\end{align}
holds. As a special case of the main result in Section~\ref{sec:Neumann} it turns out that the inequality~\eqref{eq:langweiligeUngleichung} is in fact strict, i.e.,
\begin{align}\label{eq:spannendeUngleichung}
 \lambda_k^\N < \lambda_k^\D, \qquad k \in \{1, \dots, l\}.
\end{align}
Our proof of~\eqref{eq:spannendeUngleichung} is based on an idea by Filonov in~\cite{F05} who showed an inequality for the eigenvalues of Dirichlet 
and Neumann Laplacians on a bounded domain. Its adaption to the present situation makes use of a unique continuation principle.

In fact, the inequality~\eqref{eq:spannendeUngleichung} appears as a special case of a more general result. Instead of restricting ourselves to the case of the Neumann operator $- \DNV$ we consider the selfadjoint 
operator $- \DRV$ satisfying a mixed boundary condition, namely a Robin boundary condition on a relatively open part $\omega$ of the boundary $\partial\Omega$,
\begin{align*}
 \alpha u |_\omega + \frac{\partial u}{\partial \nu} \Big|_\omega = 0
\end{align*}
for some $\alpha \in \RR$, and a Dirichlet boundary condition on the complement $\omega' = \partial \Omega \setminus \omega$. The essential spectra of $- \DDV$ and $- \DRV$ 
coincide (see Section~\ref{sec:preliminaries}) and it turns out that whenever $\omega$ is nonempty the analog of~\eqref{eq:spannendeUngleichung} for this situation is true, i.e.,
\begin{align}\label{eq:spannendeUngleichung2}
 \lambda_k^\R < \lambda_k^\D, \qquad k \in \{1, \dots, l\},
\end{align}
holds, where $\lambda_1^\R \leq \lambda_2^\R \leq \dots$ are the eigenvalues of $- \DRV$ below the bottom of the essential spectrum. The inequality~\eqref{eq:spannendeUngleichung} 
follows from~\eqref{eq:spannendeUngleichung2} setting $\omega = \partial \Omega$ and $\alpha = 0$. We remark that eigenvalue inequalities for Robin Laplacians on 
bounded domains can be found in the literature in, e.g.,~\cite{GM09,R14}.

In Section~\ref{sec:elliptic} we complement our result with an inequality for elliptic differential operators subject to Dirichlet boundary conditions with different sets of 
coefficients. For the special case of Schr\"odinger operators the result reads as follows: 
Assume that $V_1, V_2 : \Omega \to \RR$ are two bounded, measurable potentials with $V_1 (x) \leq V_2 (x)$ for all $x \in \Omega$ such that the bottoms of the essential spectra 
of $- \Delta_\D^{V_1}$ and $- \Delta_\D^{V_2}$ coincide. If, in addition, $V_1 (x) < V_2 (x)$ for all $x$ in some open ball then 
\begin{align}\label{eq:spannend?}
 \lambda_k^{\D, V_1} < \lambda_k^{\D, V_2}, \quad k \in \{1, \dots, l\},
\end{align}
whenever $- \Delta_\D^{V_2}$ (and then also $- \Delta_\D^{V_1}$) has at least $l$ eigenvalues below the bottom of the essential spectrum. The method to prove 
this observation is in line with the proofs in the previous section. We remark that for~\eqref{eq:spannend?} no regularity of the boundary of $\Omega$ is required; also the case $\Omega = \RR^d$ is included, where no boundary condition is present any more.

\subsection*{Acknowledgements}

Jussi Behrndt and Jonathan Rohleder gratefully acknowledge financial support by the Austrian Science Fund (FWF): Project P~25162-N26. The authors wish to thank Mark Ashbaugh and Fritz Gesztesy for helpful remarks and literature hints.

\section{Schrödinger operators with Dirichlet, Neumann, and Robin boundary conditions on exterior Lipschitz domains}\label{sec:preliminaries}
    
In this preparatory section we provide some preliminaries and discuss properties of Schrödinger operators with Dirichlet, Neumann, and Robin boundary conditions on exterior Lipschitz domains.
    
We assume here and in the following sections that $\Omega \subset \RR^d$, $d\geq 2$, is an unbounded open set with a compact Lipschitz boundary, i.e., $\RR^d \setminus \overline{\Omega}$ is a bounded Lipschitz domain. We require for convenience that, in addition, $\Omega$ is connected. We denote the standard Sobolev spaces on $\Omega$ and on the boundary $\partial\Omega$ by $H^s(\Omega)$, $s\in\RR$, and $H^s(\partial\Omega)$, $s\in [-1,1]$, respectively. Recall that the mapping 
\begin{equation*}
 C_0^\infty(\overline{\Omega}) \ni u \mapsto u|_{\partial\Omega}
\end{equation*}
can be extended by continuity to a bounded, surjective operator from $H^1(\Omega)$ to $H^{1/2}(\partial\Omega)$. We will use the notation $u|_{\partial\Omega}$ for the trace of $u \in H^1(\Omega)$ and we set
\begin{equation}\label{eq:H01}
 H^1_0(\Omega) := \bigl\{u\in H^1(\Omega) : u|_{\partial\Omega} = 0\bigr\}.
\end{equation}
Note that $H_0^1 (\Omega)$ coincides with the closure of $C_0^\infty (\Omega)$ in $H^1 (\Omega)$. For $u \in H^1(\Omega)$ such that $\Delta u \in L^2(\Omega)$ holds in the distributional sense the normal derivative $\normalderiv{u}$ 
is the uniquely defined element in $H^{-1/2}(\Omega)$ which satisfies Green's identity
\begin{equation}\label{eq:Green1}
 \scalar{\nabla u}{\nabla v}_{(L^2(\Omega))^d} = \LTwoScalar{-\Delta u}{v}{\Omega} + \duality{\normalderiv{u}}{v|_{\partial\Omega}}
\end{equation}
for all $v \in H^1(\Omega)$; here $(\cdot, \cdot)_{L^2 (\Omega)}$ and $(\cdot, \cdot)_{(L^2 (\Omega))^d}$ denote the inner products in $L^2 (\Omega)$ and $(L^2 (\Omega))^d$, respectively, and $\langle\cdot,\cdot\rangle$ denotes the sesquilinear duality between $H^{1/2}(\partial\Omega)$ and its dual space $H^{-1/2}(\partial\Omega)$. For the consideration of Schr\"odinger operators with mixed boundary conditions assume that $\omega$ is an open, nonempty subset of $\partial\Omega$ and set $\omega^\prime=\partial\Omega\setminus\omega$. For a function $u\in H^1(\Omega)$ we shall denote by $u |_\omega$ and $u\vert_{\omega^\prime}$ the restriction of the trace $u\vert_{\partial\Omega}$ to $\omega$ and $\omega^\prime$, respectively.

In order to introduce Schr\"{o}dinger operators with Dirichlet,  Neumann, and Robin boundary conditions let $V \in L^\infty(\Omega)$ be a real valued function and let $\alpha\in\RR$. The sesquilinear forms    
\begin{equation*}
\begin{split}
 \aDV(u,v) &= \bigl( \nabla u, \nabla v \bigr)_{(L^2(\Omega))^d} + ( Vu, v )_{L^2(\Omega)}, \\
 \dom\aDV &= H^1_0(\Omega),
\end{split}
 \end{equation*}
and
\begin{equation*}
\begin{split}
 \aRV(u,v) &= \bigl( \nabla u, \nabla v \bigr)_{(L^2(\Omega))^d} + ( Vu, v )_{L^2(\Omega)} + \alpha \bigl(u\vert_{\partial\Omega},v\vert_{\partial\Omega}\bigr)_{L^2(\partial\Omega)},\\
 \dom\aRV & = \bigl\{ u\in H^1(\Omega): u\vert_{\omega^\prime}=0\bigr\},
\end{split}
\end{equation*}
in $L^2(\Omega)$ are both densely defined, symmetric, bounded from below and closed. The corresponding semibounded, selfadjoint operators in $L^2(\Omega)$ will be denoted by $-\DDV$ and $-\DRV$ and are given by
\begin{equation}\label{dd}
\begin{split}
 -\DDV u &= -\Delta u + V u,\\
 \dom(-\DDV) & = \bigl\{u\in H^1(\Omega):\Delta u \in L^2(\Omega), u|_{\partial\Omega} = 0\bigr\},
\end{split}
\end{equation}
and
\begin{equation}\label{rr}
\begin{split}
 -\DRV u &= -\Delta u + V u,\\
 \dom (-\DRV) &= \biggl\{u\in H^1(\Omega):\Delta u \in L^2(\Omega), \alpha u\vert_{\omega} + \frac{\partial u}{\partial\nu}\Bigl|_{\omega}=0,\,u\vert_{\omega'}=0 \biggr\},
\end{split}
\end{equation}
respectively; cf., e.g.,~\cite[Chapter VII]{EE87}; here the local Robin condition for the functions in the domain of $-\DRV$ is understood in the distributional sense, namely 
$$
 \alpha u\vert_{\omega} + \frac{\partial u}{\partial\nu}\Bigl|_{\omega}=0
$$
if and only if 
$$
 \left\langle \alpha u\vert_{\partial\Omega} + \frac{\partial u}{\partial\nu}\Bigl|_{\partial\Omega},\varphi\right\rangle=0
$$
for all $\varphi\in H^{1/2}(\partial\Omega)$ such that $\varphi |_{\omega'} = 0$. The operators $- \DDV$ and $- \DRV$ satisfy the relations
\begin{align}\label{eq:darstellungDirichlet}
 \aDV (u, v) = (- \DDV u, v)_{L^2 (\Omega)}, \quad u \in \dom (- \DDV),\,\,\, v \in \dom\aDV ,
\end{align}
and
\begin{align}\label{eq:darstellungRobin}
 \aRV (u, v) = (- \DRV u, v)_{L^2 (\Omega)}, \quad u \in \dom (- \DRV),\,\,\, v \in \dom\aRV ,
\end{align}
which follow from Green's identity~\eqref{eq:Green1}. Note that in the special case $\omega = \partial \Omega$ and $\alpha = 0$ the sesquilinear form $\aRV$ coincides with the Neumann form
\begin{equation*}
\begin{split}
 \aNV(u,v) &= \bigl( \nabla u, \nabla v \bigr)_{(L^2(\Omega))^d} + ( Vu, v )_{L^2(\Omega)} ,\\
 \dom\aNV &= H^1(\Omega),
 \end{split}
\end{equation*}
and the corresponding selfadjoint operator is given by the Neumann operator
\begin{equation}\label{nn}
\begin{split}
 -\DNV u &= -\Delta u + V u,\\
 \dom(-\DNV) &= \biggl\{u\in H^1(\Omega):\Delta u \in L^2(\Omega), \frac{\partial u}{\partial\nu}\Bigl|_{\partial\Omega}= 0\biggr\}.
\end{split}
\end{equation}

The following useful proposition is known for exterior domains with smooth boundaries and $-\DNV=-\DRV$. For the convenience of 
the reader we provide a proof in the present more general situation.

\begin{proposition}\label{prop:ess_specs_coincide}
The essential spectra of $- \DDV$ and $- \DRV$ coincide.
\end{proposition}

\begin{proof}
Let $\lambda \in \CC \setminus \RR$ and consider the operators 
\begin{equation*}
 S : L^2 (\Omega) \to H^{- 1/2} (\partial \Omega),\qquad u\mapsto\frac{\partial}{\partial\nu} \big( (- \DDV - \lambda)^{-1} u \big) \Big|_{\partial \Omega},
\end{equation*}
and 
\begin{equation*}
 T : L^2 (\Omega) \to H^{1/2} (\partial \Omega),\qquad u\mapsto \big( (- \DRV - \overline{\lambda})^{-1} u \big) |_{\partial \Omega}.
\end{equation*}
It follows from the continuity of the trace and the normal derivative that both operators $S$ and $T$ are bounded. Moreover, we claim that $\ran S \subset L^2 (\partial \Omega)$ holds. Indeed, let $u \in L^2 (\Omega)$ and choose an open ball $B \subset \RR^d$ such that $\partial\Omega\subset B$. Then  $\Omega_0 := B \cap \Omega$ is a bounded Lipschitz domain with $\partial \Omega \subset \partial \Omega_0$. Let $\chi \in C^\infty (\Omega)$ be a function with $\chi = 1$ identically in a neighborhood of $\partial \Omega$ and $\chi = 0$ outside $\Omega_0$. Then the function $u_0 = (\chi (- \DDV - \lambda)^{-1} u)|_{\Omega_0}$ belongs to $H^1_0 (\Omega_0)$ and $\Delta u_0 \in L^2 (\Omega_0)$ holds. It follows from~\cite[Theorem~B]{JK95} that $u_0 \in H^{3/2} (\Omega_0)$. With the help of~\cite[Lemma~3.2]{GM11} we further conclude $\normalderiv[\partial \Omega_0]{u_0} \in L^2 (\partial \Omega_0)$. In particular,
\begin{equation*}
\begin{split}
 S u &= \frac{\partial}{\partial\nu} \big( (- \DDV - \lambda)^{-1} u \big) \Big|_{\partial \Omega}\\
 &= \frac{\partial}{\partial\nu} \big( \chi (- \DDV - \lambda)^{-1} u \big) \Big|_{\partial \Omega} + \frac{\partial}{\partial\nu} \big( (1-\chi) (- \DDV - \lambda)^{-1} u \big) \Big|_{\partial \Omega}\\
 &=   \normalderiv{u_0}
\end{split}
\end{equation*}
and hence $\ran S\subset L^2(\partial\Omega)$. By the closed graph theorem $S$ considered as an operator from $L^2 (\Omega)$ to $L^2 (\partial \Omega)$ is bounded. Since the embedding of $L^2 (\partial \Omega)$ into $H^{- 1/2} (\partial \Omega)$ is compact, it follows that $S : L^2 (\Omega) \to H^{- 1/2} (\partial \Omega)$ is compact.

Let now $u, v \in L^2 (\Omega)$ and define
\begin{align*}
 f=(- \DDV - \lambda)^{-1} u\qquad\text{and}\qquad g=(- \DRV - \overline \lambda)^{-1} v.
\end{align*}
Then we obtain with the help of~\eqref{eq:Green1} and $f\vert_{\partial\Omega}=0$
\begin{equation*}
\begin{split}
 \bigl((- \DDV - \lambda)^{-1} u - (&- \DRV - \lambda)^{-1} u, v \bigr)_{L^2(\Omega)} \\
 & = (f , v )_{L^2(\Omega)} - (u , g )_{L^2(\Omega)} \\
 & = \bigl(f ,(- \DRV - \overline \lambda)g \bigr)_{L^2(\Omega)} - \bigl((- \DDV - \lambda)f , g \bigr)_{L^2(\Omega)} \\
 & = \bigl(f ,-\Delta g \bigr)_{L^2(\Omega)}  - \bigl(-\Delta f , g \bigr)_{L^2(\Omega)} \\
 & = \left\langle\normalderiv{f},g\vert_{\partial\Omega}\right\rangle - \left\langle f\vert_{\partial\Omega},\normalderiv{g}\right\rangle\\
 & = \left\langle\frac{\partial}{\partial\nu}\bigl((- \DDV - \lambda)^{-1} u\bigr) \Big|_{\partial\Omega}, \bigl((- \DRV - \overline \lambda)^{-1} v\bigr)\vert_{\partial\Omega}\right\rangle\\
 & = \langle Su,Tv\rangle
\end{split}
\end{equation*}
and hence
\begin{align}\label{eq:resDiff}
 (- \DDV - \lambda)^{-1} - (- \DRV - \lambda)^{-1} = T^* S.
\end{align}
As $S$ is compact and $T^*$ is bounded it follows that $T^* S$ and thus the left-hand side of~\eqref{eq:resDiff} is compact. Hence the essential spectra of $- \DDV$ and $- \DRV$ coincide.
\end{proof}

\section{A strict inequality between Dirichlet and Robin eigenvalues}\label{sec:Neumann}
    
This section contains the first main result of this note. In Theorem~\ref{thm:filonov_unbeschraenkt} we shall prove a strict inequality between the eigenvalues below the essential spectrum of the Schr\"{o}dinger operators with Dirichlet and Robin boundary conditions given in \eqref{dd} and \eqref{rr}, respectively. Throughout this section $\Omega \subset \RR^d$, $d\geq 2$, is an unbounded, connected Lipschitz domain with a compact boundary and $V\in L^\infty(\Omega)$ is a real valued function. 

The following preparatory lemma is the counterpart of the lemma in~\cite{F05}, where the Laplacian on a bounded Lipschitz domain with Neumann boundary conditions was considered. In contrast to the situation in~\cite{F05}, a unique continuation principle must be employed in the proof. For the convenience of the reader we carry out the details.

\begin{lemma}\label{lemma:filonov_hilfslemma}
Let $-\DRV$ be given as in \eqref{rr} and let $\mu\in\RR$. Then
\begin{equation*}
 H^1_0(\Omega) \cap \ker\left( -\DRV - \mu \right) = \{0\}.
\end{equation*}
\end{lemma}
    
\begin{proof}
Assume that 
\begin{equation*}
 v\in H^1_0(\Omega)\cap\ker\left( -\DRV-\mu \right)
\end{equation*}
and let $\widetilde\Omega\subset \RR^d$ be an unbounded Lipschitz domain such that 
\begin{equation*}
 \Omega\subset\widetilde\Omega,\qquad \omega'\subset\partial\widetilde\Omega\quad\text{and}\quad \widetilde\Omega\setminus\Omega\not=\emptyset.
\end{equation*}
Consider the function
\begin{equation*}
 \widetilde{v}(x) := 
 \begin{cases}
  v(x), & \textrm{if } x\in\Omega, \\
  0,    & \textrm{if } x\in\widetilde\Omega\setminus\Omega,
 \end{cases}
\end{equation*}
which belongs to $H^1(\widetilde\Omega)$. Let $\widetilde{V}\in L^\infty(\widetilde\Omega)$ be the extension of $V$ by zero to $\widetilde\Omega$. Calculating the action of the distribution $(- \Delta + \widetilde V) \widetilde v$ on $\widetilde \Omega$, for each $\widetilde{\psi} \in C_0^\infty(\widetilde\Omega)$ we have
\begin{align}\label{iii}
\begin{split}
 \bigl( -\Delta \widetilde{v} + \widetilde{V} \widetilde{v}  \bigr) \bigl(\widetilde{\psi}\bigr) & = \sum\limits_{i=1}^{d}\left( \partial_i \widetilde{v} \right)\big( \partial_i \widetilde{\psi} \big) + \big( \widetilde{V}\widetilde{v}\big)\widetilde{\psi}\\
 & = \bigl(\nabla v,\nabla\overline\psi\bigr)_{(L^2(\Omega))^d}+(Vv,\overline\psi)_{L^2(\Omega)} \\
 & = \aRV (v, \overline{\psi}),
\end{split}
\end{align}
where $\psi$ is the restriction of $\widetilde{\psi}$ to $\Omega$ and $v\vert_{\partial\Omega}=0$ was used. Using \eqref{eq:darstellungRobin}, \eqref{iii}, and 
$v\in\ker(-\DRV-\mu)$ we obtain
\begin{equation*}
 \bigl( -\Delta\widetilde{v}+\widetilde{V}\widetilde{v}  \bigr) \bigl(\widetilde{\psi}\bigr) = (-\DRV v,\overline\psi)_{L^2(\Omega)}=(\mu v,\overline\psi)_{L^2(\Omega)}
 =( \mu \widetilde{v} )\bigl( \widetilde{\psi}\bigr),\quad \widetilde{\psi} \in C_0^\infty(\widetilde\Omega),
\end{equation*}
and hence $( -\Delta+\widetilde{V})\widetilde{v} = \mu \widetilde{v} \in L^2(\widetilde\Omega)$. Since $\widetilde v |_{\widetilde \Omega \setminus \Omega} = 0$, unique continuation implies  $\widetilde{v} = 0$ on $\widetilde\Omega$, see, e.g.,~\cite{W93}. Hence $v = 0$ on $\Omega$.
\end{proof}
    
Now we come to the first main result of this note. Its proof is inspired by an idea in~\cite{F05}. First we introduce some useful notation. For an interval $I \subset \RR$ which contains no essential spectrum the eigenvalue counting functions of the Dirichlet and Robin Schrödinger operator are defined by
\begin{equation}\label{nnnm}
 N_\D(I) := \dim \ran E_\D (I) \quad \textnormal{and} \quad N_\R (I) := \dim \ran E_\R (I),
\end{equation}
where $E_\D$ and $E_\R$ denote the spectral measures of $- \DDV$ and $- \DRV$, respectively, that is, $N_\D (I)$ and $N_\R(I)$ is the number of eigenvalues of $- \DDV$ and $- \DRV$, respectively, in $I$, counted with multiplicities. Recall from Proposition~\ref{prop:ess_specs_coincide} that the essential spectra of  $-\DDV$ and $-\DRV$ coincide and let 
\begin{equation}\label{mmm}
 M := \min\sigma_{\rm ess}(-\DDV) = \min\sigma_{\rm ess}(-\DRV).
\end{equation}
We then denote by
\begin{equation*}
 \lambda^\D_1 \leq \lambda^\D_2 \leq \dots < M
\end{equation*}
and
\begin{equation*}
 \lambda^\R_1 \leq \lambda^\R_2 \leq \dots < M
\end{equation*}
the discrete eigenvalues counted with multiplicities below the minimum of the essential spectrum of $-\DDV$ and $-\DRV$, respectively. Note that it follows immediately from the min-max principle for the sesquilinear forms $\aDV$ and $\aRV$ that 
\begin{align*}
 N_\R \bigl((-\infty, \mu] \bigr) \geq N_\D \bigl((-\infty, \mu]\bigr), \quad \mu < M,
\end{align*}
and that if $- \DDV$ has (at least) $l$ eigenvalues in $(- \infty, M)$ then the same holds for $- \DRV$ and
\begin{align*}
 \lambda_k^\R \leq \lambda_k^\D \qquad \text{for all}\,\,\,\,k\in\{1,\dots,l\}.
\end{align*}
The following result shows that these observations can be strengthened.
    
\begin{theorem}\label{thm:filonov_unbeschraenkt}
Let $-\DDV$ and $-\DRV$ be the Schr\"{o}dinger operators with Dirichlet and Robin boundary conditions in~\eqref{dd} and~\eqref{rr}, respectively, let $M$ be given in~\eqref{mmm}, and let $N_\D$ and $N_\R$ be the corresponding eigenvalue counting functions defined in~\eqref{nnnm}. Then for each $\mu < M$ the inequality
\begin{equation}\label{oja}
 N_\R \bigl((-\infty, \mu)\bigr) \geq N_\D \bigl((-\infty, \mu]\bigr)
\end{equation}
holds. In particular, if there exist $l$ eigenvalues of $-\DDV$ in $(-\infty,M)$ then the strict inequality 
\begin{equation}\label{bitteschoen}
 \lambda^\R_k < \lambda^\D_k
\end{equation}
holds for all $k\in\{1,\dots,l\}$.
\end{theorem}
    
\begin{proof}
 Let $\mu < M$ and recall that by the min-max-principle (see, e.g. \cite{W74}) one has
\begin{equation*}
 N_\D \bigl((-\infty, \mu]\bigr) = \max \bigl\{\hspace{-0.2mm}\dim L:L \subset \dom \aDV \textrm{ subspace, } \aDV(u,u) \leq \mu\norm{u}^2_{L^2(\Omega)}, u\in L \bigr\}
\end{equation*}
and
\begin{equation*}
 N_\R \bigl((-\infty, \mu]\bigr) = \max \bigl\{\hspace{-0.2mm}\dim L:L \subset \dom \aRV \textrm{ subspace, } \aRV(u,u) \leq \mu\norm{u}^2_{L^2(\Omega)}, u\in L \bigr\}.
\end{equation*}
Let $F$ be a subspace of $\dom \aDV=H^1_0(\Omega)$ such that $\dim F = N_\D((-\infty, \mu])$ and
\begin{equation}\label{f}
 \aDV(u,u) \leq \mu\norm{u}^2_{L^2(\Omega)} \quad \textrm{for all } u\in F.
\end{equation}
For $u\in F$ and $v\in\ker(-\DNV-\mu)$ we obtain with the help of the relations~\eqref{eq:darstellungDirichlet} and~\eqref{eq:darstellungRobin}
\begin{equation}\label{eq:filonov_unbeschraenkt_1}
\begin{split}
 \aRV(u+v,u+v) & = \aRV(u,u)+\aRV(v,v)+2\Re\,\aRV(v,u) \\
 & = \aDV(u,u)+(-\DRV v,v)_{L^2(\Omega)}+2\Re\,(-\DRV v,u)_{L^2(\Omega)} \\
 & \leq \mu\norm{u}^2_{L^2(\Omega)} +\mu\norm{v}^2_{L^2(\Omega)}+2\mu\,\Re\,(v,u)_{L^2(\Omega)} \\
 & = \mu\norm{u+v}_{L^2 (\Omega)}^2,
\end{split}
\end{equation}
where the estimate \eqref{f} was used in the third step. As $F \subset H^1_0(\Omega)$ we conclude from Lemma \ref{lemma:filonov_hilfslemma} that the sum $F \dotplus \ker(-\DRV - \mu)$ is direct. Hence it follows from \eqref{eq:filonov_unbeschraenkt_1} that
\begin{equation*}
\begin{aligned}
 N_\R \bigl((-\infty, \mu]\bigr) &\geq \dim(F) + \dim\ker\left( -\DRV-\mu \right)\\
 &= N_\D\bigl((-\infty, \mu]\bigr) + \dim\ker\left( -\DRV-\mu \right)
\end{aligned}
\end{equation*}
and this yields
\begin{equation*}
 N_\R \bigl((-\infty, \mu)\bigr) = N_\R \bigl((-\infty, \mu]\bigr) - \dim\ker\left( -\DRV-\mu \right) \geq N_\D\bigl((-\infty, \mu]\bigr),
\end{equation*}
which is \eqref{oja}. Finally, if there exist $l$ eigenvalues of the operator $-\DDV$ in $(-\infty,M)$ and $k \in \{1, \dots, l\}$ is chosen arbitrarily then~\eqref{oja} with $\mu= \lambda^\D_k$ shows $\lambda_k^\R < \lambda_k^\D$, which proves~\eqref{bitteschoen}.
\end{proof}

We immediately obtain the following corollary for the Neumann operator $- \DNV$. Here for any interval $I \subset \RR$ we write
\begin{equation}\label{nnnm2}
 N_\N (I) := \dim \ran E_\N (I),
\end{equation}
where $E_\N$ is the spectral measure of $- \DNV$. As in \eqref{mmm} we have
\begin{equation}\label{mmmm}
 M = \min\sigma_{\rm ess}(-\DDV) = \min\sigma_{\rm ess}(-\DNV)
\end{equation}
and we denote by
\begin{equation*}
 \lambda^\N_1 \leq \lambda^\N_2 \leq \dots < M
\end{equation*}
the discrete eigenvalues of $-\DNV$ below $M$, counted with multiplicities. 

\begin{corollary}\label{cor:Neumann}
Let $-\DDV$ and $-\DNV$ be the Schr\"{o}dinger operators with Dirichlet and Neumann boundary conditions in~\eqref{dd} and~\eqref{nn}, respectively, let $M$ be given in~\eqref{mmmm}, and let $N_\D$ and $N_\N$ be the corresponding eigenvalue counting functions defined in~\eqref{nnnm} and~\eqref{nnnm2}. Then for each $\mu < M$ the inequality
\begin{equation}
 N_\N\bigl((-\infty, \mu)\bigr) \geq N_\D\bigl((-\infty, \mu]\bigr)
\end{equation}
holds. In particular, if there exist $l$ eigenvalues of $-\DDV$ in $(-\infty,M)$ then the strict inequality 
\begin{equation}
 \lambda^\N_k < \lambda^\D_k
\end{equation}
holds for all $k\in\{1,\dots,l\}$.
\end{corollary}

In the next corollary we turn to the special case that the function $V$ belongs to $L^\infty (\Omega) \cap L^p (\Omega)$ for an appropriate $p$ and satisfies the growth condition 
\begin{align}\label{eq:growth}
 V(x) \leq -\alpha|x|^{-2+\varepsilon} \quad \textrm{ for } |x| > R_0
\end{align}
for some $R_0 > 0, \alpha > 0$ and $\varepsilon > 0$. In this case it can be shown as in~\cite[Example~6 in Section~XIII.4 and Problem~41]{RS78} and~\cite[Theorem XIII.6]{RS78} that the essential spectra of $- \DDV$ and $- \DRV$ equal $[0, \infty)$ and that both operators possess infinitely many negative eigenvalues. Therefore Theorem~\ref{thm:filonov_unbeschraenkt} allows the following conclusion.
    
\begin{corollary}\label{cor:slowDecay}
 Let $V\in L^\infty(\Omega)\cap L^p(\Omega)$ with $p\geq \max\{d/2,2\}$ if $d\not=4$ and $p>2$ if $d=4$, and assume that there exist constants $R_0 > 0, \alpha > 0$ and $\varepsilon > 0$ such that~\eqref{eq:growth} is satisfied. Then there exist infinitely many discrete eigenvalues of $-\DDV$ and $-\DRV$ below their essential spectrum $\sigma_{\rm ess}(-\DDV) = \sigma_{\rm ess}(-\DRV) = [0, \infty)$ and the strict inequality
\begin{equation*}
 \lambda^\R_k < \lambda^\D_k
\end{equation*}
holds for all $k \in \NN$.
\end{corollary}

\begin{remark}
The assumption in this section that $\Omega$ is connected can be dropped. In fact, Theorem~\ref{thm:filonov_unbeschraenkt} and its proof remain valid if each connected component $\Lambda$ of $\Omega$ satisfies $\partial \Lambda \cap \omega \neq \emptyset$. In particular, Corollary~\ref{cor:Neumann} remains true also if $\Omega$ is not connected.
\end{remark}

\section{A remark on eigenvalue inequalities for elliptic operators with varying coefficients}\label{sec:elliptic}
    
In this section we turn to the related subject of eigenvalue inequalities for pairs of elliptic operators with different coefficients and a fixed boundary condition. 
For simplicity we restrict ourselves to a Dirichlet boundary condition; similar results can be proved for Neumann, Robin or mixed boundary conditions as well. 
In this section we require only that $\Omega \subset \RR^d$, $d \geq 2$, is a nonempty, open, connected set, without assuming any regularity or compactness of 
the boundary. Also the case $\Omega = \RR^d$ is included. We make use of the space $H_0^1 (\Omega)$, which is defined to be the closure of 
$C_0^\infty (\Omega)$ in $H^1 (\Omega)$; if the boundary of $\Omega$ is sufficiently smooth, e.g., Lipschitz, then $H_0^1 (\Omega)$ coincides with 
the kernel of the trace operator; cf.~\eqref{eq:H01}.

Let $\cL_1, \cL_2$ be second order differential expressions on $\Omega$ of the form
\begin{align*}
 \cL_i = - \sum_{j, k = 1}^d \partial_j a_{jk, i} \partial_k + a_i,
\end{align*}
where $a_{jk, i} : \overline \Omega \to \CC$ are bounded Lipschitz functions and $a_i : \Omega \to \RR$ are bounded and measurable, $i = 1, 2$. The differential expressions are assumed to be formally symmetric, i.e., $a_{jk, i} (x) = \overline{a_{kj, i} (x)}$ for all $x \in \overline{\Omega}$, $i = 1, 2$, and uniformly elliptic, i.e., there exist $E_i > 0$ with
\begin{align*}
 \sum_{j, k = 1}^d a_{jk, i} (x) \xi_j \xi_k \geq E_i \sum_{k = 1}^d \xi_k^2, \quad x \in \overline{\Omega}, \quad \xi = (\xi_1, \dots, \xi_d)^\top \in \RR^d, \quad i = 1, 2.
\end{align*}
The selfadjoint Dirichlet operators associated with $\cL_i$ in $L^2 (\Omega)$ are given by
\begin{align}\label{eq:ADirichlet}
 A_i u = \cL_i u, \quad \dom A_i = \left\{ u \in H_0^1 (\Omega) : \cL_i u \in L^2 (\Omega) \right\}, \qquad i = 1, 2.
\end{align}
They correspond to the densely defined, symmetric, semibounded, closed sesquilinear forms
\begin{align*}
 \fa_i (u, v) = \sum_{j, k = 1}^d \int_\Omega a_{jk, i} \partial_k u \,\overline{\partial_j v} d x + \int_\Omega a_i \,u \,\overline v d x, \quad \dom \fa_i = H_0^1 (\Omega), \quad i = 1, 2,
\end{align*}
that is,
\begin{align}\label{abcd}
 \fa_i (u, v) = (A_i u, v)_{L^2 (\Omega)}, \quad u \in \dom A_i,\,\,\, v \in H_0^1 (\Omega), \quad i = 1, 2.
\end{align}
    
In the following we focus on the case that the infima of the essential spectra of $A_1$ and $A_2$ coincide. For instance, this is the case if the coefficients of the difference $\cL_2 - \cL_1$ are close to zero outside a compact set in an appropriate sense. If $\Omega$ is bounded or, more generally, has finite Lebesgue measure, then the essential spectra of both operators are empty. We define
\begin{align}\label{eq:Mneu}
 M := \inf \sigma_{\rm ess} (A_1) = \inf \sigma_{\rm ess} (A_2),
\end{align}
including the possibility $M = + \infty$ if the essential spectra are empty. Moreover, we assume that $\cL_1$ and $\cL_2$ are ordered in the sense that
\begin{align}\label{eq:order}
 \sum_{j, k = 1}^d a_{jk, 1} (x) \xi_j \overline{\xi_k} \leq \sum_{j, k = 1}^d a_{jk, 2} (x) \xi_j \overline{\xi_k}, \quad x \in \overline \Omega, \,\,\xi = (\xi_1, \dots, \xi_d)^\top \in \RR^d,
\end{align}
(i.e., the matrix $(a_{jk, 2} (x) - a_{jk, 1} (x))_{j, k}$ is nonnegative for all $x \in \overline \Omega$), and
\begin{align}\label{eq:order2}
 a_1 (x) \leq a_2 (x), \quad x \in \Omega.
\end{align}
These two conditions immediately imply 
\begin{align}\label{eq:formIneq}
 \fa_1 (u, u) \leq \fa_2 (u, u), \quad u \in H_0^1 (\Omega).
\end{align}
In particular, if $A_2$ possesses at least $l$ eigenvalues in $(- \infty, M)$ then the same holds for $A_1$ and 
\begin{align}\label{eq:ineqL1L2}
 \lambda_k (A_1) \leq \lambda_k (A_2), \quad k \in \{1, \dots, l \},
\end{align}
where $\lambda_1 (A_i) \leq \lambda_2 (A_i) \leq \dots$ denote the eigenvalues of $A_i$ in $(- \infty, M)$, counted with multiplicities, $i = 1, 2$. The following observation shows that the inequality~\eqref{eq:ineqL1L2} is strict whenever the coefficients of $\cL_1$ and $\cL_2$ differ sufficiently strongly. For each interval $I \subset (- \infty, M)$ we denote by $N_i (I)$ the number of eigenvalues of $A_i$ in $I$, counted with multiplicities, $i = 1, 2$.

\begin{theorem}\label{thm:A1A2}
Assume that $\inf \sigma_{\rm ess} (A_1) = \inf \sigma_{\rm ess} (A_2)$, let $M$ be given in~\eqref{eq:Mneu} and let the assumptions~\eqref{eq:order}--\eqref{eq:order2} be satisfied. Moreover, assume that there exists an open ball $\cO \subset \Omega$ such that at least one of the following conditions is satisfied:
\begin{enumerate}
 \item[(a)] $a_1 (x) < a_2 (x)$ for all $x \in \cO$ or
 \item[(b)] the matrix $(a_{jk, 2} (x) - a_{jk, 1} (x))_{j, k}$ is invertible for all $x \in \cO$.
\end{enumerate}
Then for all $\mu < M$ the inequality
\begin{align}\label{eq:countingFunctions}
 N_1 \big( (- \infty, \mu) \big) \geq N_2 \big( (- \infty, \mu] \big)
\end{align}
holds. In particular, if there exist $l$ eigenvalues of $A_2$ in $(- \infty, M)$ then 
\begin{align*}
 \lambda_k (A_1) < \lambda_k (A_2)
\end{align*}
holds for all $k\in\{1,\dots,l\}$.
\end{theorem}

\begin{proof}
Let $\mu < M$. Similar to the proof of Theorem~\ref{thm:filonov_unbeschraenkt} we can choose a subspace $F \subset H_0^1 (\Omega)$ such that $\dim F = N_2 ((-\infty, \mu])$ and
\begin{equation}\label{eq:schaetzab}
 \fa_2 (u,u) \leq \mu\norm{u}^2_{L^2(\Omega)} \quad \textrm{for all } u\in F.
\end{equation}
For $u\in F$ and $v\in\ker(A_1 - \mu)$ we obtain with the help of \eqref{abcd} and \eqref{eq:formIneq}
\begin{equation}\label{eq:filonov_unbeschraenkt_2}
\begin{split}
 \fa_1 (u+v,u+v) & = \fa_1 (u,u)+\fa_1 (v,v)+2\Re\,\fa_1 (v,u) \\
 & \leq \fa_2 (u,u) + (A_1 v,v)_{L^2(\Omega)}+2\Re\,(A_1 v,u)_{L^2(\Omega)} \\
 & \leq \mu\norm{u}^2_{L^2(\Omega)} +\mu\norm{v}^2_{L^2(\Omega)}+2\mu\,\Re\,(v,u)_{L^2(\Omega)} \\
 & = \mu\norm{u+v}_{L^2 (\Omega)}^2.
\end{split}
\end{equation}
Moreover, the sum $F \dotplus \ker(A_1 - \mu)$ is direct. Indeed, if $w \in F \cap \ker (A_1 - \mu)$ then $\fa_1 (w, w) = \mu \|w\|_{L^2 (\Omega)}^2$ and thus~\eqref{eq:formIneq} and~\eqref{eq:schaetzab} yield
\begin{align*}
 \fa_1 (w, w) = \fa_2 (w, w),
\end{align*}
that is, 
\begin{align*}
 \int_\Omega ( \cA \nabla w, \nabla w )_{\CC^d} d x + \int_\Omega (a_2 - a_1) |w|^2 d x = 0,
\end{align*}
where $\cA (x) = (a_{jk, 2} (x) - a_{jk, 1} (x))_{j,k}$ for $x \in \Omega$. Since $\cA (x)$ is a nonnegative matrix and $a_2 (x) - a_1 (x) \geq 0$ for all $x \in \Omega$ by the assumptions~\eqref{eq:order} and~\eqref{eq:order2}, it follows
\begin{align}\label{eq:beidenull}
 ( \cA \nabla w, \nabla w )_{\CC^d} = 0 \quad \text{and} \quad (a_2 - a_1) |w|^2 = 0
\end{align}
identically on $\Omega$. If the condition (a) of the theorem holds for some open ball $\cO \subset \Omega$ then the second identity in~\eqref{eq:beidenull} implies $w |_\cO = 0$ and since $\cL_1 w = \mu w$ a unique continuation principle yields $w = 0$ on $\Omega$, see, e.g.,~\cite{W93}. If the condition (b) is satisfied then the first equality in~\eqref{eq:beidenull} leads to $\nabla w = 0$ on $\cO$ so that $w = c$ identically on $\cO$ for some constant $c \in \CC$ and unique continuation implies $w = c$ identically on $\Omega$. Since $w \in H_0^1 (\Omega)$ it follows again $w = 0$ identically. Hence the sum $F \dotplus \ker (A_1 - \mu)$ is direct and from~\eqref{eq:filonov_unbeschraenkt_2} we obtain
\begin{equation*}
\begin{aligned}
 N_1 \bigl((-\infty, \mu]\bigr) &\geq \dim(F) + \dim\ker\left( A_1 - \mu \right) = N_2 \bigl((-\infty, \mu]\bigr) + \dim\ker\left( A_1 - \mu \right),
\end{aligned}
\end{equation*}
which proves~\eqref{eq:countingFunctions}.
\end{proof}

For the special case of Schr\"odinger differential operators on an exterior domain Theorem~\ref{thm:A1A2} reads as follows; cf.\ the remarks above Corollary~\ref{cor:slowDecay}.

\begin{corollary}
Let $\Omega$ be a connected open set which is the exterior of a bounded domain or equals $\RR^d$. Furthermore, let $V_1, V_2 \in L^\infty(\Omega)\cap L^p(\Omega)$ with $p\geq \max\{d/2,2\}$ if $d\not=4$ and $p>2$ if $d=4$ be real valued functions and let $A_1$ and $A_2$ denote the selfadjoint Dirichlet operators corresponding to the differential expressions $\cL_1 = - \Delta + V_1$ and $\cL_2 = - \Delta + V_2$ as in~\eqref{eq:ADirichlet}. Assume that $V_1 \leq V_2$ on $\Omega$ and that there exists an open ball $\cO \subset \Omega$ such that $V_1 < V_2$ on $\cO$. Then for all $\mu < M$ the inequality
\begin{align*}
 N_1 \big( (- \infty, \mu) \big) \geq N_2 \big( (- \infty, \mu] \big)
\end{align*}
holds. In particular, if there exist $l$ eigenvalues of $A_2$ in $(- \infty, M)$ then 
\begin{align*}
 \lambda_k (A_1) < \lambda_k (A_2)
\end{align*}
holds for all $k\in\{1,\dots,l\}$.
\end{corollary}

\begin{remark}
The assumption that $\Omega$ is connected can be dropped if it is assumed that each connected component of $\Omega$ contains an 
open ball $\cO$ such that one of the conditions (a) or (b) of Theorem~\ref{thm:A1A2} holds.
\end{remark}

\end{document}